\documentclass[a4paper, 11pt]{article}
\usepackage[utf8]{inputenc}
\usepackage[T1]{fontenc}
\usepackage{tgpagella}
\usepackage{amsmath, amssymb, amsthm, bm, xcolor}
\usepackage{algorithm}
\usepackage{algpseudocode}

\usepackage[textwidth = 15.2cm]{geometry}
\definecolor{main_color}{HTML}{2789d6}
\usepackage[pdfa, hidelinks, colorlinks = true, allcolors = main_color]{hyperref}

\newtheorem{theorem}{Theorem}
\newtheorem{definition}{Definition}
\newtheorem{lemma}{Lemma}
\newtheorem{proposition}{Proposition}

\DeclareMathOperator{\rank}{rank}
\DeclareMathOperator{\lin}{lin}
\DeclareMathOperator{\diag}{diag}
\DeclareMathOperator{\conv}{conv}
\DeclareMathOperator{\supp}{supp}
\DeclareMathOperator{\ssum}{sum}
\DeclareMathOperator{\sign}{sign}
\DeclareMathOperator{\colsum}{colsum}
\DeclareMathOperator{\rowsum}{rowsum}

\newcommand{\R}{\mathbb{R}}
\newcommand{\Q}{\mathbb{Q}}
\newcommand{\Z}{\mathbb{Z}}
\newcommand{\N}{\mathbb{N}}
\newcommand{\bA}{\bm{A}}
\newcommand{\bx}{\bm{x}}
\newcommand{\bB}{\bm{B}}
\newcommand{\ba}{\bm{a}}
\newcommand{\br}{\bm{r}}
\newcommand{\by}{\bm{y}}
\newcommand{\bz}{\bm{z}}
\newcommand{\bc}{\bm{c}}
\newcommand{\bb}{\bm{b}}
\newcommand{\bv}{\bm{v}}
\newcommand{\bw}{\bm{w}}
\newcommand{\bt}{\bm{t}}
\newcommand{\bd}{\bm{d}}
\newcommand{\bW}{\bm{W}}
\newcommand{\bC}{\bm{C}}
\newcommand{\bT}{\bm{T}}
\newcommand{\bU}{\bm{U}}
\newcommand{\bH}{\bm{H}}
\newcommand{\blambda}{\bm{\lambda}}
\newcommand{\poly}{\mathcal{P}}
\newcommand{\face}{\mathcal{F}}

\title{A Threshold Number for the Shortest Vector Problem in the Infinity Norm}

\author{Stefan Kuhlmann\thanks{ETH Zürich, Department of Mathematics, Institute for Operations Research, Switzerland, \texttt{\{stefan.kuhlmann,robert.weismantel\}@ifor.math.ethz.ch}} \and Robert Weismantel\footnotemark[1]}

\date{ }

\begin{document}
	\maketitle
	\begin{abstract}
		\noindent
		For an integer full column rank matrix $\bA$, we consider the lattice that consists of all integer combinations of columns of $\bA$. 
		We prove that a shortest non-zero vector $\bA\bz$ has infinity norm equal to $1$ whenever the number of columns of $\bA$ is at least $\Delta$, the largest absolute value of a full rank subdeterminant of $\bA$. This structural result allows us to design a fixed-parameter tractable algorithm in $\Delta$ for computing a shortest lattice vector in the infinity norm. It also has several applications in integer optimization. In particular, for a polyhedron defined by $\bA\bx\leq \bb$ with integer-valued $\bb$, an optimal integer solution lies on a face whose dimension is at most $\Delta - 1$. This paper is a continuation of work presented at IPCO 2026 \cite{svp_delta}.
	\end{abstract}
	
	\section{The Shortest Vector Problem}
	\label{sec_intro}
	A fundamental algorithmic problem in the geometry of numbers with numerous applications to other areas of mathematics and computer science is the shortest lattice vector problem. 
	This problem is easy to state. Let $\bA \in \Z^{m \times n}$ be a matrix of full column rank. The lattice $\bA\Z^n$ consists of all integral combinations of the columns of $\bA$, i.e., 
	$\bA\Z^n := \left\{ \by \in \Z^m : \by = \bA\bx \text{ for } \bx \in \Z^n\right\}.$
	Given such a matrix $\bA \in \Z^{m \times n}$ of full column rank and a norm $\| \cdot\|$, a natural question is to determine a non-zero vector in the lattice $\bA\Z^n$ that attains the smallest norm: $\min_{\bz\in\Z^n\backslash\lbrace \bm{0}\rbrace}\Vert\bA\bz\Vert$.
	In this paper, the underlying norm is the infinity norm, which is given by $\| \bx \|_\infty  := \max\{|x_1|, |x_2|,\ldots,|x_n|\}$ for
	$\bx = (x_1,\ldots,x_n)^\top \in \R^n$. Then the shortest vector problem in the infinity norm is as follows
	\begin{align}
		\tag{SVP}
		\label{svp}
		\min_{\bz\in\Z^n\backslash\lbrace \bm{0}\rbrace} \Vert \bA\bz \Vert_\infty.
	\end{align}
	It has been shown by van Emde Boas that \eqref{svp} is NP-hard \cite{vanEmdeBoas81}. Hardness and algorithmic results for \eqref{svp} and approximate versions of \eqref{svp} play an important role in lattice-based cryptography; see \cite{Micciancio2009} for a comprehensive survey on the topic. 
	Many algorithms for solving the shortest lattice vector problem exactly or approximately have been devised
	over the past decades. It is beyond the scope of this paper to discuss the literature in detail. Rather let us refer to several milestones and references on the topic.
	
	It began with the remarkable work of Lenstra, Lenstra, Lov\'asz \cite{LLL}, who introduced the notion of a ``reduced basis''. One important property of such a basis is that its first vector is approximately a shortest lattice vector with respect to the Euclidean norm. This provides a bound on the length of a shortest vector that can be utilized  by an exhaustive search procedure to solve \eqref{svp} in the Euclidean norm in running time $2^{\mathcal{O}(n^3)}$. Kannan \cite{Kannan2,Kannan1} refined this approach and improved the running time for computing a shortest lattice vector significantly to $2^{\mathcal{O}(n \log n)}$; see also \cite{A22,A26,A28,A48} for further modifications and improvements of Kannan's algorithm.
	The first randomized algorithm to compute a shortest lattice vector in single exponential time in the Euclidean norm was developed by Ajtai, Kumar, Sivakumar \cite{AjtaiKumarSivakumar}. The authors introduced the technique of ``randomized sieving''. 
	Roughly speaking, their method samples exponentially many lattice points and combines them to generate shorter vectors with positive probability; see \cite{Eisenbrand,Regev} for excellent surveys on this topic.
	The randomized sieving technique has been the subject of intensive investigation. In recent years, it has been modified, generalized, improved and can be applied to arbitrary norms. It gives exact randomized algorithms for \eqref{svp} that run in single exponential time and require single exponential space; see \cite{BlomerNaewe,MV2,A49,A50}, \cite{AggarwalMukhopadhyay} for the fastest algorithm in the infinity norm, and \cite{Aggarwaletal,AggarwalDadushRegevStephens-Davidowitz2015,AggarwalStephens-Davidowitz2017} for state-of-the-art results in the Euclidean norm based on ``discrete Gaussian sampling''. 
	Another appealing approach to solve \eqref{svp} in the Euclidean norm was proposed by Micciancio and Voulgaris \cite{MV1}. Their method is based on the ``Voronoi cell'' of a lattice. It provides us with the first deterministic single exponential time algorithm for \eqref{svp} in the Euclidean norm.
	
	The methods to tackle \eqref{svp} discussed so far all measure complexity in terms of the dimension $n$. In this paper, we ask the following: can we say something more when we also fix the parameter $\Delta$ defined to be the largest absolute value among all full rank subdeterminants of the given matrix $\bA$? In other words, we fix a constant $\Delta$ and work with lattices defined by $\Delta$-modular matrices. This is made formal below.	
	\begin{definition}
		\label{def_delta}
		A matrix $\bA \in \Z^{m \times n}$ of full column rank is called $\Delta$-modular if $\left|\det\bB\right|\leq\Delta$ for all full rank submatrices $\bB$ of $\bA$ and there exists at least one full rank submatrix $\bB$ of $\bA$ such that $\left|\det\bB\right|=\Delta$.
	\end{definition}
	If a lattice $\bA\Z^n$ has a basis given by a $\Delta$-modular matrix, then this property is intrinsic to the lattice, since changing the basis of $\bA\Z^n$ corresponds to applying a unimodular transformation from the right, see \cite[Section 4]{schrijvertheorylinint86}, which preserves full rank subdeterminants up to a sign. 
	
	It is a major open problem in integer programming whether a linear discrete optimization problem is solvable in polynomial time when the underlying matrix is $\Delta$-modular and $\Delta$ is constant; see \cite{nearly_tu,artweiszenbimodalgo2017,polynomial_ring,two_nonzeros_per_row,two_non_zeros_most_rows,naegele2024advancesstrictlydelta,naegelesanzencongruence2023} for some partial progress in this direction. Inspired by this question, we ask ourselves whether \eqref{svp} is solvable in polynomial time when the lattice $\bA\Z^n$ is determined by a $\Delta$-modular matrix $\bA$ and $\Delta$ is constant. We are not aware of any FPT algorithm for \eqref{svp} parameterized solely by $\Delta$ prior to this work. The only available results also parameterize by $m-n$ \cite{fpt_gribanov,fpt_gribanov_2}.  
	Our algorithmic result below is possible because of a threshold phenomenon:
	if $n\geq \Delta$, then for a $\Delta$-modular matrix $\bA$ a shortest non-zero lattice vector $\bA\bz$ attains the smallest possible value $\|\bA\bz\|_{\infty}=1$. 
	To make this more formal, let $\mathcal{M}_{\Delta}$ denote the set of $\Delta$-modular matrices with full column rank. We consider the value
	\begin{align}
		\label{supremum}
		\sup_{\bA\in\mathcal{M}_{\Delta}}\left\lbrace n\in \N :  \Vert \bA\bz \Vert_\infty \geq 2 \; \forall \bz\in\Z^n\backslash\lbrace \bm{0}\rbrace \right\rbrace.
	\end{align}
	It turns out that one can replace the supremum in \eqref{supremum} by a maximum. 
	\begin{theorem}
		\label{thm_shortest_vector}
		Let $\Delta\in \N_{\geq 1}$. Then
		\begin{align*}
			\sup_{\bA\in\mathcal{M}_{\Delta}}\left\lbrace n\in \N :  \Vert \bA\bz \Vert_\infty \geq 2 \; \forall \bz\in\Z^n\backslash\lbrace \bm{0}\rbrace \right\rbrace = \Delta - 1.
		\end{align*}
	\end{theorem}
	Theorem~\ref{thm_shortest_vector} can be interpreted as a variant of Minkowski's convex body theorem parameterized in $\Delta$: Define $\mathcal{Q}:=\lbrace \bx\in\R^n : -\bm{1}\leq\bA\bx\leq\bm{1}\rbrace$. Whenever $n\geq \Delta$, the convex body $\mathcal{Q}$ contains a non-zero integer vector.
	
	The proof of Theorem~\ref{thm_shortest_vector} is constructive. It gives rise to a polynomial time algorithm for fixed $\Delta$ that either computes the shortest non-zero lattice vector in the lattice $\bA \Z^n$ with respect to the infinity norm or finds a full rank submatrix of $\bA$ with determinant larger than $\Delta$ in absolute value.
	
	\begin{theorem}
		\label{thm_svp_algorithm}
		Let $\Delta\in\N_{\geq 1}$, $\bA\in\Z^{m\times n}$ of full column rank, and suppose that $n\geq \Delta$. Then one can find $\bz^*\in\Z^n\backslash\lbrace \bm{0}\rbrace$ such that $\Vert\bA\bz^*\Vert_\infty = 1$ or return a full rank submatrix $\bB$ of $\bA$ with $\left|\det\bB\right| > \Delta$. This can be done
		in $\mathcal{O}(m^2n^2\Delta^2)$ arithmetic operations.
	\end{theorem}
	Theorem~\ref{thm_svp_algorithm} implies an FPT algorithm for \eqref{svp} parameterized by $\Delta$: Given some matrix $\bA\in\Z^{m\times n}$, one can first compute the Hermite normal form in strongly polynomial time for fixed $\Delta$, cf. \cite{DomichKannanTrotter1987}, to verify that $\bA$ has full column rank or restrict to a submatrix of full column rank. The next step is to check whether this submatrix has rank at most $\Delta - 1$ or not. In the latter case, one runs the algorithm that provides a proof of Theorem~\ref{thm_svp_algorithm}. Otherwise, one can use an algorithm that solves \eqref{svp} in polynomial time when $n$ is a constant. 
	It will become evident in Section~\ref{sec_algorithm} that the FPT algorithm uses polynomial space. 
	The exponents in the running time stated in Theorem~\ref{thm_svp_algorithm} can be improved by incorporating fast algorithms for Gauss-Jordan elimination and matrix multiplication. In fact, Theorem~\ref{thm_svp_algorithm} does not require $\Delta$ to be constant. It always applies when $\Delta \leq n$.
	
	In comparison to the results in the extended abstract \cite[Theorem 1]{svp_delta}, we present a precise value for \eqref{supremum}. The insight that leads to an improved bound in Theorem~\ref{thm_shortest_vector} is based on Camion bases that have been studied in the context of discrete geometry. 
	\begin{definition}\label{def_camionbasis}
		A full rank submatrix $\bB$ of $\bA$ is a Camion basis if
		there exist signs $\sigma_1,\ldots,\sigma_n\in\lbrace -1,1\rbrace$
		such that each row of $\bA\cdot\bB^{-1}\cdot
		\diag(\sigma_1,\ldots,\sigma_n)$ is either non-negative or
		non-positive.
	\end{definition}
	Camion showed that such a full rank submatrix always exists
	\cite{Camion1968}. This result was generalized by Bland and Cho \cite{BlandCho1987}.
	To present their
	result, let us introduce the following notation. For a matrix $\bA\in\R^{m\times n}$ and $I\subseteq\lbrack m\rbrack$, denote by $\bA_I$ the submatrix consisting of the rows indexed by $I$. Also, define $\mathcal{B} := \lbrace I\subseteq \lbrack m\rbrack : \det\bA_I \neq 0\rbrace$. We say that
	$\bv\in \R^n$ is in general position for $\bA$ and $\mathcal{B}$ if
	$(\bA^{-\top}_{I}\bv)_j\neq 0$ for all $j\in\lbrack n\rbrack$ and
	$I\in\mathcal{B}$. Observe that, for each $\bA$ and $\mathcal{B}$,
	there always exist vectors in general position since vectors that are
	not in general position have to lie in finitely many subspaces.
	Moreover, for $I\in\mathcal{B}$, we denote by $\by^{I} \in \R^m$ the
	unique vector that gives a solution to $\bA^\top \by = \bv$ with
	$y^{I}_i = 0$ for all $i\notin I$. So, for $\lbrace i_1 < \cdots <
	i_n\rbrace = I \in \mathcal{B}$, the vector $\by^{I}$ satisfies
	\begin{align*}
		y^{I}_{i_j} = (\bA^{-\top}_{I}\bv)_j \text{ for } j\in\lbrack n\rbrack.
	\end{align*}
	We present the result by Bland and
	Cho as stated in \cite[Proposition~1]{BlandKoSturmfels1993}.
	
	\begin{proposition}[{\cite[Theorem~3.5]{BlandCho1987}}]
		\label{prop_bland_cho}
		Let $\bA\in\R^{m\times n}$ have full column rank. Let  $\bv$ be in general position for
		$\bA$ and $\mathcal{B}$. If $\by^{I}$ is a vertex of $\conv\lbrace
		\by^{J} : J\in\mathcal{B}\rbrace$, then $\bA_I$ is a Camion basis of
		$\bA$.
	\end{proposition}
	To prove Theorem~\ref{thm_shortest_vector}, we have to modify Proposition~\ref{prop_bland_cho} to account for determinants of the involved bases. This is made precise in Lemma~\ref{lemma_camion_property}. One essential ingredient for the algorithm in Theorem~\ref{thm_svp_algorithm} is a modified version of a polynomial time algorithm that finds a Camion basis due to Fukuda and Musitelli \cite{FukudaMusitelli2006}. The precise link is given in Section~\ref{sec_algorithm}.

	\section{Implications for Integer Optimal Solutions}
	\label{sec_app_sparse}
		Theorem~\ref{thm_shortest_vector} has applications beyond algorithmic results. It can be used to derive novel results in the theory of integer programming. 
	For a polyhedron $\poly := \lbrace\bx\in\R^n : \bA\bx\leq\bb\rbrace$, Theorem~\ref{thm_shortest_vector} implies that each vertex of the integer hull of $\poly$ lies on a face of $\poly$ whose dimension is bounded by $\Delta - 1$.
	
	\begin{theorem}
		\label{thm_vertices_integer_hull}
		Let $\bA\in\Z^{m\times n}$ be $\Delta$-modular, $\bb\in\Z^m$, and $\poly := \lbrace\bx\in\R^n : \bA\bx\leq\bb\rbrace$. Then the vertices of the convex hull of $\poly\cap\Z^n$ lie on faces of $\poly$ of dimension at most $\Delta - 1$.
	\end{theorem}
	If we set $\Delta = 1$ in Theorem~\ref{thm_vertices_integer_hull}, the matrix $\bA$ is unimodular and we recover that vertices of the convex hull of $\poly\cap\Z^n$ coincide with vertices of $\poly$, a fact that follows from a result by Hoffman and Kruskal \cite{hoffkruskal1956tustatement}. When $\Delta = 2$, the vertices of the integer hull lie on faces of dimension $0$ or $1$. This also follows from a known result due to Veselov and Chirkov \cite{veselovchirkovbimodular09}. For $\Delta \geq 3$, upper bounds of this type were not known previously. In particular, for $\Delta = 3$, Theorem~\ref{thm_vertices_integer_hull} ensures the existence of integer vectors on faces of dimension at most $2$, provided that $\poly\cap\Z^n\neq \emptyset$. When $\Delta = 3$, it remains an open problem whether one can find such an integer vector or decide that no such integer vector exists in polynomial time. 
	
	Theorem~\ref{thm_vertices_integer_hull} can also be applied to obtain novel upper bounds on the number of non-zero entries of optimal solutions of integer optimization problems in standard form. 
	For this to make sense, we assume that $\bA\in \Z^{m \times n}$ has full row rank. We call $\bA$ $\Delta$-modular if $\bA^\top$ is $\Delta$-modular according to Definition~\ref{def_delta}. Given such a matrix $\bA$ and  $\bb\in\Z^m$, $\bc\in\Z^n$, an integer optimization problem in standard form is of the form 
	
	\begin{equation}
		\label{ilp_standard_form}
		\max\left\{ \bc^\top\bx \colon \bA\bx = \bb, \, \bx\geq \bm{0}, \,\bx \in \Z^n\right\}.
	\end{equation}
	It is an important question to detect an optimal solution of smallest support, i.e., an optimal solution $\bz^*$ with the smallest number of strictly positive entries among $z^*_1,\ldots,z^*_n$.
	
	The problem of bounding the support of solutions for integer optimization problems has been studied quite intensively in the past decade. A first bound in terms of $m$ and determinants of $\bA$ is due to Aliev et al. \cite[Theorem 1]{alideloeisoerweissupportint2018}. Their result gives an upper bound on the support of an optimal integer solution of  $m + \log_2(\sqrt{\det\bA\bA^\top})$; see also \cite[Theorem 3]{alievAverkovLoeraOertel21} for related results. 
	A problem with these bounds, for our purposes, is that the expression $\det\bA\bA^\top$ cannot be bounded solely by the parameter $\Delta$, the largest full rank subdeterminant of $\bA$.
	For instance, the matrix $\bT\in\Z^{(n-1)\times \binom{n}{2}}$ obtained from deleting a row of the oriented incidence matrix of a complete graph satisfies $\det\bT\bT^\top = n^{n-2}$ by Kirchhoff's matrix tree theorem; see \cite[Section 7.2]{Cvetković_Rowlinson_Simić_2009}. However, $\bT$ is totally unimodular and thus there always exist integer solutions with at most $n-1$ non-zero entries, assuming feasibility. 
	Nevertheless, using the Cauchy-Binet formula on $\det\bA\bA^\top$, bounds in terms of $m$ and $\Delta$ on the support of an optimal integer solution are available; see \cite[Theorem 4]{lee2020improving} and \cite[Section 3.1, Remark 5]{gribanov2024delta}. The latter gives an upper bound of $2.81\cdot m + 1.18\cdot\log_2(\Delta)$. 
	Let us also mention that, when the columns of $\bA$ positively span $\R^m$, one can improve this to $2m + \log_2(\Delta)$ \cite[Theorem 2]{alievAverkovLoeraOertel21} for feasible integer solutions.
	
	Below, we establish the first upper bound  on the support of an optimal integer solution of the form $m$ plus a function of $\Delta$. More precisely, our bound  on the support of an optimal integer solution is $m + \Delta - 1$. Given that there exists an optimal solution to the LP relaxation of \eqref{ilp_standard_form} with support at most $m$, the support bounds for the continuous and the integer optimal solution differ essentially by this threshold value $\Delta - 1$, independent of $m$.
	To the best of our knowledge, prior to the result below, upper bounds of the form $m$ plus a function of $\Delta$ are only available when $m=1$ \cite[Theorem 1.2]{alievdeloesparelindio2017} or in an asymptotic setting \cite{oertel2020distributions}. For $\bx\in\R^n$, we define $\supp(\bx) := \lbrace i\in\lbrack n\rbrack: x_i \neq 0\rbrace$.
	\begin{theorem}
		\label{thm_support_bound}
		Let $\bA \in \Z^{m \times n}$ be $\Delta$-modular, $\bb \in \Z^m$, and $\bc \in \R^n$. 
		If (\ref{ilp_standard_form}) has an optimal solution, then there exists an optimal solution $\bz^*$ such that 
		\begin{align*}
			\left|\supp(\bz^*)\right| \leq m + \Delta - 1.
		\end{align*}
	\end{theorem}
	This gives the best-known upper bound on the sparsity of optimal integer solutions when $\Delta$ is constant or small compared to $m$. 
	Similar to Theorem~\ref{thm_vertices_integer_hull}, Theorem~\ref{thm_support_bound} extends known results for $\Delta\in\lbrace 1,2\rbrace$ to arbitrary values of $\Delta$. 
	We also provide a first non-trivial lower bound on the number of non-zero entries of integer solutions in terms of $m$ and $\Delta$. This bound even holds for totally $\Delta$-modular matrices, i.e., matrices in which all $k\times k$ submatrices have a determinant bounded by $\Delta$ in absolute value for all $k\in\lbrack m\rbrack$.
	
	\begin{proposition}
		\label{prop_lower_bound_sparsity}
		Let $\Delta\in\N_{\geq 1}$, $m := (\Delta - 1)^2 + 1$, and $n := (\Delta - 1)^2 + \Delta$. There exists a totally $\Delta$-modular matrix $\bA\in\Z^{m\times n}$ with full row rank and $\bb\in\Z^m$ such that 
		\begin{align*}
			\left|\supp(\bz^*)\right| = m + \Delta - 1
		\end{align*}
		for all optimal integer solutions $\bz^*$ of (\ref{ilp_standard_form}) with respect to all $\bc\in\R^n$. 
	\end{proposition}
	For the sake of completeness, let us briefly mention that there is also a line of research investigating upper bounds on the support of optimal solutions of \eqref{ilp_standard_form} in terms of $m$ and $\Vert\bA\Vert_\infty$, the largest entry of $\bA$ in absolute value. In this regime, a first upper bound is due to Eisenbrand and Shmonin \cite[Theorem 1]{eisenbrandshmonincaratheodorybounds06}, which has been improved to $2m\log_2(2\sqrt{m}\Vert\bA\Vert_\infty)$ by Aliev et al. \cite[Theorem 1]{alideloeisoerweissupportint2018}. This bound is known to be tight, up to the constant $2$ in front of $m$ and the logarithm \cite{berndt2021new,berndt2024new}. Hence, in contrast to Theorem~\ref{thm_support_bound}, bounds of the form $m$ plus a function of $\Vert\bA\Vert_\infty$ cannot exist.
	
	\section{Proof of Theorem~\ref{thm_shortest_vector}}
	\label{sec_proof_threshold}
	Given a $\Delta$-modular matrix $\bA$ with sufficiently large
	rank, our main goal is to find an integer combination of the columns
	with coefficients given by $\bz$ such that $\Vert\bA\bz\Vert_\infty =
	1$. One hurdle is that although all full rank subdeterminants of $\bA$
	are bounded by $\Delta$, the entries of $\bA$ may be arbitrarily large,
	and, therefore, each column of $\bA$ may be arbitrarily large in the
	infinity norm. This makes it challenging to construct an integer vector
	$\bA\bz$ of small norm. To circumvent this issue, we apply a common
	normalization technique when dealing with $\Delta$-modular matrices: we
	analyze subdeterminants of $\bA\cdot\bB^{-1} \in \Q^{m\times n}$, where
	$\bB$ is some invertible full rank submatrix of $\bA$.
	The following lemma establishes that columns of
	$\bA\cdot\bB^{-1}$ have a bounded infinity norm. 
	Let the rows of $\bA$ be denoted by $\ba^\top_i$ for $i \in\lbrack m\rbrack$.
	\begin{lemma}
		\label{lemma_submatrix_determinants}
		Let $\bA\in\Z^{m\times n}$ have full column rank. Let $I\subseteq
		\lbrack m\rbrack$ such that $\bA_I$ is invertible with $I = \lbrace i_1 < \ldots < i_n\rbrace$ and $i\in\lbrack m\rbrack$, $j\in\lbrack n\rbrack$. Then
		\begin{align*}
			\left|(\bA\cdot\bA^{-1}_{I})_{i,j}\right| =
			\frac{\left|\det(\ba_{i_1},\ldots,\ba_{i_{j-1}},\ba_i,\ba_{i_{j+1}},\ldots,\ba_{i_n})\right|}{\left|\det\bA_I\right|}.
		\end{align*}
		If $\bA$ is $\Delta$-modular and $\left|\det\bA_I\right| = \Delta$, we have $\bA\cdot\bA^{-1}_I\in\lbrack -1,1\rbrack^{m\times n}$.
	\end{lemma}
	\begin{proof}
		The equation follows by applying Laplace expansion along $\ba_i$ to the determinant in the numerator. To prove the second property, we use this equation and $\left|\det\bA_I\right| = \Delta$ to get
		\begin{align*}
				\left|(\bA\cdot\bA^{-1}_{I})_{i,j}\right| =
			\frac{\left|\det(\ba_{i_1},\ldots,\ba_{i_{j-1}},\ba_i,\ba_{i_{j+1}},\ldots,\ba_{i_n})\right|}{\Delta} \leq 1.
		\end{align*}
	\end{proof}

	Observe that $\bA\cdot\bB^{-1}\in\Q^{m\times n}$ may not have
	integer entries anymore. We compensate for this by picking a submatrix
	$\bB$ that makes $\bA\cdot\bB^{-1}\in\Q^{m\times n}$ easy to analyze.
	In particular, we claim that there exists a submatrix $\bB$ and signs
	$\sigma_1,\ldots,\sigma_n\in\lbrace -1,1\rbrace$ such that
	\begin{enumerate}
		\item
		\label{two_properties_1}
		$\bA\cdot\bB^{-1}\in\lbrack -1,1\rbrack^{m\times n}$
		and
		\item
		\label{two_properties_2}
		$\bA\cdot\bB^{-1}
		\cdot\diag(\sigma_1,\ldots,\sigma_n)$ contains no row with entries $1$ and $-1$.
	\end{enumerate}
	These properties will be enough to prove the upper bound in Theorem~\ref{thm_shortest_vector}. We have to argue why such a submatrix exists. 
	The existence of a submatrix satisfying Property~\ref{two_properties_1} is easily verified. Take
	any submatrix $\bB$ of $\bA$ such that   $\left| \det\bB\right| = \Delta$. This
	implies $\bA\cdot\bB^{-1}\in\lbrack -1,1\rbrack^{m\times n}$ by
	Lemma~\ref{lemma_submatrix_determinants}.
	To obtain Property~\ref{two_properties_2}, we use tools from the theory of Camion bases. 
	In particular, we prove a variant of Proposition~\ref{prop_bland_cho} restricted to full rank submatrices with determinant $\Delta$ in absolute value. Although a Camion basis would already imply Property~\ref{two_properties_2}, we cannot guarantee Property~\ref{two_properties_1} since the absolute value of the determinant of the basis might not equal $\Delta$. This explains why we must prove a variant of Proposition~\ref{prop_bland_cho} that takes the determinant of the basis in absolute value into account.
	
	For a $\Delta$-modular matrix $\bA\in \Z^{m\times n}$, let
	$\mathcal{B}_\Delta := \lbrace I\subseteq \lbrack m\rbrack : \left|\det
	\bA_I\right| = \Delta\rbrace\subseteq \mathcal{B}$. A vector $\bv$ is in
	general position for $\bA$ and $\mathcal{B}_{\Delta}$ if $(\bA^{-\top}_{I}\bv)_j\neq 0$ for all $j\in\lbrack n\rbrack$ and
	$I\in\mathcal{B}_{\Delta}$.
	
	\begin{lemma}
		\label{lemma_camion_property}
		Let $\bA\in\Z^{m\times n}$ be $\Delta$-modular. Let  $\bv$ be in general
		position for $\bA$ and $\mathcal{B}_{\Delta}$. If $\by^{I}$ is a vertex
		of $\conv\lbrace \by^{J} : J\in\mathcal{B}_{\Delta}\rbrace$, then there exist
		signs $\sigma_1,\ldots,\sigma_n\in\lbrace -1,1\rbrace$ such that no
		row of $\bA\cdot\bA^{-1}_{I}\cdot\diag(\sigma_1,\ldots,\sigma_n)$
		contains entries $1$ and $-1$.
	\end{lemma}
	
	\begin{proof}
		Let $I = \lbrace i_1 <\cdots < i_n\rbrace$ and
		choose $\sigma_j := \sign(y^{I}_{i_j})$ for all $j\in\lbrack
		n\rbrack$. Each $\sigma_j$ is indeed contained in $\lbrace -1,1\rbrace$
		because $\bv$ is in general position, which implies $y^{I}_{i_j} =
		(\bA^{-\top}_{I}\bv)_{j} \neq 0$ for all $j\in\lbrack n \rbrack$.
		
		Suppose the claim is not true, that is, there exists a
		row of $\bA\cdot\bA^{-1}_{I}\cdot\diag(\sigma_1,\ldots,\sigma_n)$ that
		contains entries $1$ and $-1$. This implies  that there exists a row of $\bA\cdot\bA^{-
			1}_{I}$ indexed by $k\in\lbrack m\rbrack$ and
		corresponding column indices $p,q\in\lbrack n\rbrack$ such that
		\begin{align}
			\label{proof_camion_property_crossing}
			\left(\ba_k^\top \bA^{-1}_{I}\right)_p =
			\sigma_p = \sign\left(y^{I}_{i_p}\right) \text{ and }\left(\ba_k^\top
			\bA^{-1}_{I}\right)_q = - \sigma_q = - \sign\left(y^{I}_{i_q}\right).
		\end{align}
		We use row $k$ and the two rows $i_p$ and $i_q$
		that correspond in $\bA_I$ to the columns with indices $p$ and $q$ to construct a
		new full rank submatrix with determinant $\Delta$ in absolute value.
		Let $I_p := I\backslash \lbrace i_p\rbrace \cup \lbrace k\rbrace$ and
		$I_q:= I\backslash \lbrace i_q\rbrace \cup \lbrace k\rbrace$. By
		Lemma~\ref{lemma_submatrix_determinants}, we have  $I_p,I_q \in
		\mathcal{B}_{\Delta}$. For the remainder of the proof, we show that
		$\by^{I}$ is a convex combination of $\by^{I_p}$ and $\by^{I_q}$. This
		contradicts the assumption that $\by^{I}$ is a vertex of $\conv\lbrace
		\by^{J} : J\in\mathcal{B}_{\Delta}\rbrace$ and proves the claim.
		
		Let $\bd \in \R^m$ such that $d_k = 1$, $d_{i_j} = -
		(\bA^{-\top}_{I}\ba_k)_j$ for $j\in \lbrack n\rbrack$ and $d_i = 0$ for
		$i\notin I\cup \lbrace k\rbrace$. By construction, we have $\bA^\top\bd
		= \ba_k + \bA_{I}^\top \bd_{I} = \ba_k -\ba_k = \bm{0}$. This implies
		that $\bA^\top (\by^I +\lambda \bd) = \bv$ for all $\lambda\in \R$. In
		particular, there exist $\lambda_p, \lambda_q\in \R$ such that
		$\by^{I_p} = \by^{I} + \lambda_p \bd$ and $\by^{I_q} = \by^{I} +
		\lambda_q \bd$. So $\by^{I}, \by^{I_p}, \by^{I_q}$ lie on an affine
		line. Hence, to show that $\by^{I}$ is a convex combination of the
		other two, it suffices to prove that $\by^I$ lies between the other
		two, which is implied by $\lambda_p > 0$ and $\lambda_q < 0$. Using
		that $0 = y^{I_p}_{i_p} = y^{I}_{i_p} + \lambda_p d_{i_p}$, we obtain
		\begin{align*}
			\lambda_p = -\frac{y^{I}_{i_p}}{d_{i_p}} = -
			\frac{y^{I}_{i_p}}{-(\bA^{-\top}_{I}\ba_k)_p} = 
			\frac{y^{I}_{i_p}}{(\bA^{-\top}_{I}\ba_k)_p}= \left|
			y^{I}_{i_p}\right| > 0,
		\end{align*}
		where the last equation follows from
		\eqref{proof_camion_property_crossing} and the strict inequality follows from
		the general position of $\bv$. Accordingly, one obtains that 
		\begin{align*}
			\lambda_q = -\frac{y^{I}_{i_q}}{d_{i_q}} = -
			\left|y^I_{i_q} \right| < 0.
		\end{align*}
	\end{proof}
	
	\begin{proof}[Proof of Theorem~\ref{thm_shortest_vector}]
		We begin by showing that
		$\sup_{\bA\in\mathcal{M}_{\Delta}}\left\lbrace n\in \N :  \Vert \bA\bz
		\Vert_\infty \geq 2 \; \forall \bz\in\Z^n\backslash\lbrace
		\bm{0}\rbrace \right\rbrace$ is at most $\Delta -1$ and then provide an
		example that it is at least $\Delta - 1$.
		
		Let $\bA$ be $\Delta$-modular and $n\geq \Delta$. Our
		goal is to construct a non-zero integer vector $\bz^*$ such that
		$\Vert\bA\bz^*\Vert_{\infty} = 1$. Pick a full rank submatrix $\bA_I$ and signs $\sigma_1,\ldots,\sigma_n$ such that Properties~\ref{two_properties_1} and \ref{two_properties_2} hold. Such a submatrix exists because Lemma~\ref{lemma_camion_property} ensures the existence of a submatrix $\bA_{I'}$ that satisfies Property~\ref{two_properties_2} and $\det\bA_{I'} = \pm \Delta$. The latter implies Property~\ref{two_properties_1} by Lemma~\ref{lemma_submatrix_determinants}.
		Note that our choice $\bA_I$ may not have determinant $\Delta$ in absolute value. 
		
		Let $\bB^{-1} := \bA^{-1}_{I}\cdot\diag(\sigma_1,\ldots,\sigma_n)$ and  $\br_1,\ldots,\br_n$
		be the columns of $\bB^{-1}$.
		Consider the lattice $\Lambda := \bB^{-1}\Z^n$. Since
		$\bm{I}_n = \bB^{-1}\cdot \bB$, we have $\Z^n\subseteq \Lambda$. So
		$\Lambda/\Z^n$ is a finite abelian group and $|\Lambda/\Z^n| =
		(\det\Lambda)^{-1} = \left|\det\bB\right|\leq \Delta$.
		We analyze the group structure of $\Lambda/\Z^n$ given
		by the columns $\br_1,\ldots,\br_n$ that map to group elements $\br_1 +
		\Z^n, \ldots, \br_n + \Z^n$ in $\Lambda/\Z^n$.
		
		Suppose $\br_i + \Z^n = \Z^n$ for some $i\in\lbrack
		n\rbrack$. This is equivalent to $\br_i\in\Z^n$. Under this assumption,
		we get $-\bm{1}\leq \bA\br_i\leq \bm{1}$ by
		Property~\ref{two_properties_1}. This settles this case.
		
		From now on we can assume that $\br_i + \Z^n\neq \Z^n$ for all
		$i\in\lbrack n\rbrack$. Using this, $|\Lambda/\Z^n| \leq \Delta$, and
		$n\geq \Delta$, there exist $i,j\in\lbrack n\rbrack$ with $i\neq j$
		such that $\br_i + \Z^n = \br_j + \Z^n$ by the pigeonhole principle.
		This is equivalent to $\br_i - \br_j \in \Z^{n}$. By
		Property~\ref{two_properties_1}, we have $-2\bm{1}\leq
		\bA(\br_i-\br_j)\leq 2 \bm{1}$. An upper or lower bound is tight only if
		there exists a row in $\bA\cdot \bB^{-1}$ that contains entries $1$ and $-1$.
		However, our choice of $\bB$ excludes this due to
		Property~\ref{two_properties_2}. So we get $-2\bm{1} < \bA(\br_i-
		\br_j) < 2\bm{1}$. Integrality of $\bA$ and $\br_i-\br_j$ implies that
		$-\bm{1}\leq \bA(\br_i-\br_j)\leq \bm{1}$ and we can set $\bz^* :=
		\br_i-\br_j$. This finishes the proof of the upper bound.
		
		It remains to present a lower bound of $\Delta - 1$.
		Let $D := (V,A)$ be the directed graph given by nodes
		$V := \lbrace 1,\ldots, \Delta\rbrace$ and arcs $A := \lbrace (i,j) :i
		< j\text{ for }i,j\in\lbrack \Delta\rbrack\rbrace$.
		Let $\bT'$ be the arc-node incidence matrix of $D$.
		Consider the matrix $\bT$ that arises from $\bT'$ by deleting the last
		column.
		Define the invertible matrix
		\begin{align*}
			\bB := \begin{pmatrix}
				1 & & & \\
				& \ddots & & \\
				& & 1 & \\
				\Delta - 1 & \cdots & \Delta - 1 &
				\Delta
			\end{pmatrix} \in\Z^{(\Delta - 1)\times
				(\Delta - 1)}
		\end{align*}
		and let $\bA := \bT\cdot\bB$. Since
		$\bT\in\Z^{\binom{\Delta}{2}\times (\Delta - 1)}$ is totally unimodular
		and has full column rank, cf. \cite[Chapter
		19]{schrijvertheorylinint86}, the matrix $\bA$ has full column rank as
		well and is $\Delta$-modular. We claim that
		$\min_{\bz\in\Z^n\backslash\lbrace
			\bm{0}\rbrace}\Vert\bA\bz\Vert_\infty \geq 2$, which proves the result.
		
		Let $\blambda = (\lambda_1,\ldots,\lambda_{\Delta -
			1})^\top \in \Z^{\Delta - 1}$ and $\bz := \bB^{-1}\blambda \in
		\Z^{\Delta - 1}$. We have $\bA\bz = \bT\blambda$. Observe that $\bT$
		contains the unit matrix, which is given by the rows that correspond to
		the arcs incident to the vertex $\Delta$, whose corresponding column we
		deleted. Therefore we obtain $\Vert\bA\bz\Vert_\infty\geq
		\Vert\blambda\Vert_\infty$. Hence, it suffices to study the case when
		$\Vert\blambda\Vert_\infty = 1$. One can calculate that
		\begin{align*}
			\bB^{-1} = \begin{pmatrix}
				1 & & & \\
				& \ddots & & \\
				& & 1 & \\
				-\frac{\Delta - 1}{\Delta} & \cdots &
				-\frac{\Delta - 1}{\Delta} & \frac{1}{\Delta}
			\end{pmatrix}.
		\end{align*}
		This gives the following characterization for integrality
		\begin{align*}
			\bB^{-1}\blambda \in \Z^{\Delta - 1}
			\Leftrightarrow \lambda_1 + \ldots + \lambda_{\Delta - 1} \equiv 0 \mod
			\Delta.
		\end{align*}
		So we have $\lambda_1 + \ldots + \lambda_{\Delta - 1}
		= 0$ since $\left|\lambda_1 + \ldots + \lambda_{\Delta - 1}\right| \leq
		\Delta - 1$ by $\Vert \blambda\Vert_\infty = 1$. Therefore, there exist
		two indices $i,j\in\lbrack \Delta - 1\rbrack$ such that $\lambda_i = 1$
		and $\lambda_j = -1$. Without loss of generality $i < j$. Then the row $\bt^\top$ of $\bT$ that corresponds to
		the arc $(i,j)$ gives $\bt^\top \blambda = \pm2$, which implies
		$\Vert\bA\bz\Vert_\infty =\Vert\bT\blambda\Vert_\infty\geq 2$.
	\end{proof}
		
	\section{Proofs of the Polyhedral Results}
	\label{sec_standard_form_proofs}
	To prove Theorem~\ref{thm_vertices_integer_hull}, we need to apply Theorem~\ref{thm_shortest_vector} to lower-dimen\-sional subspaces. More precisely, we need to investigate vectors satisfying $-\bm{1}\leq\bA\bx\leq \bm{1}$ and $\bA_I\bx = \bm{0}$ for some $I\subseteq\lbrack m\rbrack$ such that the rows of $\bA_I$ are linearly independent. It turns out that Theorem~\ref{thm_shortest_vector} already applies in this setting. We discuss this here first, before proving Theorem~\ref{thm_vertices_integer_hull}. 
	
	Let $\bA\in\Z^{m\times n}, \Delta \in \N_{\geq 1}$, and $I\subseteq \lbrack m\rbrack$ be a set of indices of linearly independent rows of $\bA$. We define
	\begin{align*}
		\Delta_I := \max\left\{ \left|\det\left(\begin{array}{c}
				\bA_{I}\\
				\bA_{J}
			\end{array}\right)\right|:J\subseteq\left[m\right],\;\left|J\right|=n-\left|I\right|\right\}
	\end{align*}
	and the set of all full rank matrices that contain a subset of linearly independent rows indexed by $I$ such that $\Delta_I = \Delta$:
	\begin{align*}
			\overline{\mathcal{M}_{\Delta}} := \left\lbrace (\bA, I) : \bA\in\Z^{m\times n} \;\text{full column rank}, I \subseteq \lbrack m\rbrack, \rank\bA_I = |I|, \Delta_I = \Delta\right\rbrace.
		\end{align*}
	Now we extend the value in \eqref{supremum} to subspaces
	\begin{align*}
		\sup_{(\bA,I)\in\overline{\mathcal{M}_{\Delta}}}\left\lbrace \dim(\ker\bA_I)\in\N :  \Vert \bA\bz \Vert_\infty \geq 2 \; \forall \bz\in\ker\bA_I\cap\Z^n\backslash\lbrace \bm{0}\rbrace \right\rbrace.
	\end{align*}
	We apply Theorem~\ref{thm_shortest_vector} to get the following result.
	\begin{lemma}
		\label{lemma_subspaces_to_fullrank}
		Let $\Delta \in \N_{\geq 1}$. Then 
			\begin{align*}
			\sup_{(\bA,I)\in\overline{\mathcal{M}_{\Delta}}}\left\lbrace \dim(\ker\bA_I)\in\N :  \Vert \bA\bz \Vert_\infty \geq 2 \; \forall \bz\in\ker\bA_I\cap\Z^n\backslash\lbrace \bm{0}\rbrace \right\rbrace = \Delta - 1.
		\end{align*}
	\end{lemma}
	\begin{proof}
		Theorem~\ref{thm_shortest_vector} applied to $(\bA,\emptyset)\in\overline{\mathcal{M}_{\Delta}}$ already implies that the supremum is at least $\Delta - 1$. To prove the upper bound of $\Delta - 1$, it suffices to investigate elements of $\overline{\mathcal{M}_{\Delta}}$ with $I \neq \emptyset$. 
		Let $(\bA,I)\in \overline{\mathcal{M}_{\Delta}}$ be such that $\dim(\ker\bA_I)\geq \Delta$. Our goal is to show the existence of $\bz^* \in \ker\bA_I\cap\Z^n\backslash\lbrace \bm{0}\rbrace$ such that $\Vert\bA\bz^*\Vert_\infty = 1$. To achieve this, we change the representation of $\bA$. 
		Let $\bU\in\Z^{n\times n}$ be a unimodular matrix such that $\bA_I\cdot\bU$ is in Hermite normal form; see \cite[Chapter~4]{schrijvertheorylinint86} for more on Hermite normal forms. Let without loss of generality $\bA_I$ be the first $|I|$ rows of $\bA$. Then the transformation by $\bU$ implies that 
		\begin{align*}
			\bA\cdot\bU = \begin{pmatrix}
					\bH & \bm{0} \\ 
					\star & \bA'
				\end{pmatrix},
		\end{align*}
		where $\bA_I\cdot\bU = (\bH, \bm{0})$ for an invertible matrix $\bH\in\Z^{|I|\times |I|}$. The advantage of this new representation is that we have $\ker(\bA_I\cdot\bU) = \ker (\bH,\bm{0}) = \lbrace 0\rbrace^{|I|}\times \R^{n - |I|}$. In other words, elements in $\ker(\bA_I\cdot \bU)\cap\Z^n\backslash\lbrace \bm{0}\rbrace$ have the first $|I|$ coordinates equal to $0$. Therefore, we consider the last $n - |I|$ coordinates and the matrix $\bA'\in\Z^{(m-|I|)\times (n - |I|)}$. The latter is of full column rank since $\bA$ has full column rank. Let $\delta := \Delta/\left|\det\bH\right|$. Then $\bA'$ is also $\delta$-modular since $\Delta_I = \Delta$ as the multiplication by $\bU$ does not change the absolute value of full rank determinants. So we established that $\bA'\in\mathcal{M}_{\delta}$. 
		Moreover, as $n - |I| = \dim(\ker\bA_I)\geq \Delta\geq \delta$, Theorem~\ref{thm_shortest_vector} ensures that there exists $\bz'\in\Z^{n-|I|}\backslash\lbrace \bm{0}\rbrace$ with $\Vert\bA'\bz'\Vert_\infty = 1$. Let $\bz \in \ker(\bA_I\cdot\bU)\cap\Z^n\backslash\lbrace \bm{0}\rbrace$ be the vector $\bz'$ with $|I|$ zeros prepended. We have $\Vert \bA\cdot\bU \bz\Vert_\infty = 1$ by construction. Now set $\bz^* := \bU \bz\in\Z^n\backslash\lbrace \bm{0}\rbrace$. Then $\Vert\bA\bz^*\Vert_\infty = 1$ and $\bA_I\bz^* = \bA_I\cdot\bU\bz = \bm{0}$.
	\end{proof}
	The key property of $\Delta_I$ for proving Theorem~\ref{thm_vertices_integer_hull} is that $\Delta_I\leq \Delta$ by definition.
	
	\begin{proof}[Proof of Theorem~\ref{thm_vertices_integer_hull}]
		We will show the contraposition of the claim, i.e., every integer vector in $\poly\cap\Z^n$ that is not contained in a face of $\poly$ with dimension at most $\Delta - 1$ is not a vertex of the convex hull of $\poly\cap\Z^n$. 
		Let $\by \in \poly\cap \Z^n$ be a vector that does not lie on a face of $\poly$ of dimension at most $\Delta - 1$. We give a description of the minimal face $\mathcal{F}$ of $\poly$ that contains $\by$ in its relative interior: Let $K\subseteq\lbrack m\rbrack$ be the index set for the tight inequalities of $\bA\by\leq \bb$, that is, $\bA_K\by = \bb_K$. Select a set $I\subseteq K$ such that the rows of $\bA_I$ are linearly independent and $\rank\bA_I = \rank\bA_K$. Then $\face = \poly \cap \lbrace \bx\in\R^n : \bA_I\bx = \bb_I\rbrace$.  
		Since $\by$ does not lie on a face of $\poly$ of dimension at most $\Delta - 1$, we have $\dim(\ker\bA_I)=\dim(\mathcal{F})\geq \Delta\geq\Delta_I$. 
		By Lemma~\ref{lemma_subspaces_to_fullrank}, there exists $\bz^*\in\Z^n\backslash\lbrace \bm{0}\rbrace$ such that $-\bm{1}\leq \bA\bz^*\leq \bm{1}$ and $\bA_I\bz^* = \bm{0}$. 
		This implies that $\by \pm \bz^* \in \face \cap \Z^n\subseteq \poly\cap\Z^n$. We conclude that $1/2\cdot(\by + \bz^*) + 1/2\cdot(\by - \bz^*) = \by$. Hence, $\by$ is a convex combination of integer vectors in $\poly$ and thus not a vertex of the integer hull. 
	\end{proof}
	
	Our next goal is to apply Theorem~\ref{thm_vertices_integer_hull} to obtain Theorem~\ref{thm_support_bound}. 
	This involves switching between inequality form and standard form while preserving determinants. 
	To do so, we use a known result concerning orthogonal lattice bases stated below; see \cite[Theorem 4.2]{oxley20222} for a proof and \cite{veselov1980estimates} for an earlier reference, as cited in \cite{gribanov2024delta}. 
	The result also follows from a classical identity which relates the Plücker coordinates of a Grassmannian to the Plücker coordinates of the dual Grassmannian; cf., for instance, \cite[Book III, Chapter XIV, Theorem I]{Hodge_Pedoe_1994}. 
	In the statement, we write $\bar{I}:=\lbrack n\rbrack \backslash I$ for the complement of $I\subseteq\lbrack n\rbrack$ and $\bA_{\cdot,I}$ to indicate the submatrix of $\bA$ restricted to columns indexed by $I$. Also, the expression $\gcd\bA$ denotes the greatest common divisor of the full rank subdeterminants of $\bA$. Note that, for the remainder of this section, $\bA\in\Z^{m\times n}$ is a matrix of full row rank instead of full column rank.
	
	\begin{lemma}
		\label{lemma_orthogonal_lattices}
		Let $\bA\in\Z^{m\times n}$ have full row rank. Let $\bW\in\Z^{n\times (n-m)}$ have full column rank such that $\bA\cdot \bW = \bm{0}$. Then we have
		\begin{align*}
			1/\gcd\bA\cdot\left|\det\bA_{\cdot, I}\right| = 1/\gcd\bW\cdot\left|\det\bW_{\bar{I},\cdot}\right| 
		\end{align*}
		for all $I\subseteq \lbrack n\rbrack$ with $|I| = m$.
	\end{lemma}
	Let us briefly discuss this result. Denote by $\ba_1^\top,\ldots,\ba_m^\top$ the rows of $\bA$, similarly by $\bw_1,\ldots,\bw_{n-m}$ the columns of $\bW$, and let $L:=\lin\lbrace \ba_1,\ldots,\ba_m\rbrace$ be the linear space spanned by $\ba_1,\ldots,\ba_m$. The assumption $\bA\cdot\bW = \bm{0}$ states that the linear space $\lin\lbrace\bw_1,\ldots,\bw_{n-m}\rbrace$ is orthogonal to $L$, that is, $L^\perp = \lin\lbrace\bw_1,\ldots,\bw_{n-m}\rbrace$. Observe that Lemma~\ref{lemma_orthogonal_lattices} applies to all bases of $L^\perp$ given by integer vectors. This gives us the freedom to choose a suitable basis.  
	In the following proof, we select $\bw_1,\ldots,\bw_{n-m}$ to be a basis of the lattice $L^\perp\cap\Z^n$, which implies $\gcd\bW = 1$ and therefore that $\bW$ is a $(\Delta/\gcd\bA)$-modular matrix provided that $\bA$ is $\Delta$-modular. 
	
	\begin{proof}[Proof of Theorem~\ref{thm_support_bound}]
		Let $\bz^*$ be an optimal solution of \eqref{ilp_standard_form} such that $\bz^*$ is a vertex of the integer hull of $\mathcal{S}:= \lbrace \bx\in\R^n: \bA\bx=\bb,\bx\geq \bm{0} \rbrace$. Consider the polyhedron $\bz^* - \mathcal{S} = \lbrace \bx\in\R^n : \bA\bx = \bm{0}, \bx\leq\bz^*\rbrace$. Observe that $\bm{0}$ is a vertex of the integer hull of $\bz^* - \mathcal{S}$ as $\bz^*$ is a vertex of the integer hull of $\mathcal{S}$. Let $L$ be the linear space spanned by the rows of $\bA$. We apply Lemma~\ref{lemma_orthogonal_lattices}: Choose $\bW\in\Z^{n\times (n-m)}$ such that the columns of $\bW$ form a basis of $L^\perp\cap\Z^n$. By Lemma~\ref{lemma_orthogonal_lattices}, the matrix $\bW$ satisfies $\gcd\bW = 1$ and is a $(\Delta/\gcd\bA)$-modular matrix. We set $\delta := \Delta/\gcd\bA$. 
		Consider the polyhedron $\poly := \lbrace\by\in\R^{n-m}: \bW\by\leq\bz^*\rbrace$. The vectors in $\poly$ correspond one-to-one to vectors in $\bz^* - \mathcal{S}$. More precisely, 
		the linear map defined by $\by \mapsto \bW\by$ is an isomorphism that maps $\poly$ onto $\bz^*-\mathcal{S}$ and its restriction to $\Z^{n-m}$ maps one-to-one to $L^\perp\cap\Z^n$ as $\gcd\bW = 1$.
		Therefore, $\bm{0}\in\poly$ is a vertex of the integer hull of $\poly$. From Theorem~\ref{thm_vertices_integer_hull}, it follows that $\bm{0}$ lies on a face of $\poly$ of dimension at most $\delta - 1$. So there are at least $(n - m) - (\delta - 1)$ tight inequalities, indexed by elements in $I\subseteq\lbrack n\rbrack$, such that $\bm{0} = \bW_I \bm{0} = \bz^*_I$. We use this and $\delta\leq \Delta$ to get
		\begin{align*}
			\left|\supp(\bz^*)\right|\leq n - \left|I\right| \leq n - \left((n - m) - (\delta - 1)\right) = m + \delta - 1\leq m + \Delta - 1.
		\end{align*}
	\end{proof}
	It is possible to use Lemma~\ref{lemma_orthogonal_lattices} to transfer the construction of the lower bound in Theorem~\ref{thm_shortest_vector} into standard form. However, this requires some minor technical modifications. Hence we present below a concrete example and give an ad hoc proof that no non-negative integer solution with fewer than $m + \Delta - 1$ non-zero entries exists.
	
	\begin{proof}[Proof of Proposition~\ref{prop_lower_bound_sparsity}]		
		Let $m = (\Delta - 1)^2 + 1$ and $n = m + \Delta - 1$. By $\bm{I}_l$, we denote the $l\times l$ unit matrix and $\bm{1}_l$ is the all-ones vector with $l$ entries for $l\in\N$. Consider the following system of linear equations
		\begin{align*}
			\underbrace{\begin{pmatrix}
					\bm{I}_{\Delta - 1} & \bm{I}_{\Delta - 1} &  &  \\
					\bT & &  \bm{I}_{m - \Delta} &  \\ 
					-\bm{1}_{\Delta - 1}^\top &  &  & \Delta		
			\end{pmatrix}}_{=:\bA}\bx = \underbrace{\begin{pmatrix}
					2\cdot\bm{1}_{\Delta - 1} \\ 
					\bm{1}_{m - \Delta} \\
					1
			\end{pmatrix}}_{=:\bb},
		\end{align*}
		where $\bT\in \Z^{(m - \Delta)\times (\Delta - 1)}$ is the arc-node incidence matrix of a complete directed graph with $\Delta - 1$ nodes. First observe that $\bA\in\Z^{m\times n}$ has full row rank as the last $m$ columns give an invertible submatrix. Next we claim that the matrix $\bA$ is totally $\Delta$-modular: 
		Since $\bT$ is an arc-node incidence matrix, we know that $\bT$ is totally unimodular; cf. \cite[Chapter 19]{schrijvertheorylinint86}. Adding unit vectors as rows and then unit vector columns to $\bT$ preserves total unimodularity. So we obtain that the submatrix $\bA'$ given by the first $m-1$ rows of $\bA$ is totally unimodular. Consider an $m\times m$ submatrix $\bB$ of $\bA$ that contains the last column. Applying Laplace expansion along the last column and the fact that $\bA'$ is totally unimodular tells us that the determinant of $\bB$ equals $0$ or $\pm \Delta$. Similarly, if $\bB$ is an $m\times m$ submatrix of $\bA$ that does not contain the last column of $\bA$, we can apply Laplace expansion along the last row of $\bB$ and obtain that $\left|\det\bB\right|\leq \Delta - 1$ as $\bA'$ is totally unimodular. We established that $\bA$ is $\Delta$-modular. To see that $\bA$ is totally $\Delta$-modular, just append the last unit vector to $\bA$ and observe that the new matrix remains $\Delta$-modular. However, since the resulting matrix contains a unit matrix and is $\Delta$-modular, all $k\times k$ subdeterminants are bounded by $\Delta$ in absolute value for $k\in\lbrack m\rbrack$.
		
		We continue with showing that no sparse integer-valued solution exists. 
		The all-ones vector $\bm{1}$ is a non-negative integer solution to $\bA\bx=\bb$. For the remainder of this section, we show that there is no other non-negative integer solution, which implies the claim as $\left|\supp(\bm{1})\right| = n = m + \Delta - 1$. 
		Suppose that $\bz\in\Z^{n}_{\geq 0}$ satisfies $\bA\bz = \bb$. Our first claim is that $z_n = 1$. Consider the last equation of $\bA\bz = \bb$, which states 
		\begin{align*}
			\Delta\cdot z_n - 1= z_1 + \ldots + z_{\Delta - 1}.
		\end{align*}
		If $z_n = 0$, then the right hand side has to be negative, which is not possible. So we have $z_n\geq 1$. If $z_n\geq 2$, we get that $2\Delta - 1 \leq z_1 + \ldots + z_{\Delta - 1}$. By averaging, we know that there has to be an index $i\in\lbrack \Delta - 1\rbrack$ such that $z_i\geq 3$. Consider the $i$-th equation of $\bA\bz = \bb$. This equation tells us
		\begin{align*}
			2 = z_i + z_{\Delta - 1 + i} \geq  3 + z_{\Delta - 1 + i},
		\end{align*}
		which implies that $z_{\Delta - 1 + i}$ is negative, a contradiction. So we conclude that $z_n = 1$. Our next claim is that $z_i \geq 1$ for all $i\in\lbrack \Delta - 1\rbrack$. Suppose that $z_i = 0$ for some $i\in\lbrack \Delta - 1\rbrack$. Since $\Delta - 1 = z_1 + \ldots + z_{\Delta - 1}$ by the last equation, we get from averaging again that there exists $j\in\lbrack \Delta - 1\rbrack$ such that $z_j\geq 2$. Consider the equation $z_j - z_i + z_k = 1$ for some $k\in\lbrace 2\Delta - 1,\ldots, n -1\rbrace$, which corresponds to the rows of $\bA$ that contain the arc-node incidence matrix $\bT$. We get
		\begin{align*}
			1 = z_j - z_i + z_k \geq 2 + z_k,
		\end{align*}
		a contradiction to $z_k$ being non-negative. Hence, we deduce that $z_i\geq 1$. This holds for all $i\in\lbrack \Delta - 1\rbrack$. As $\Delta - 1 = z_1 + \ldots + z_{\Delta - 1}\geq \Delta - 1$, we have $z_i = 1$ for all $i\in\lbrack \Delta - 1\rbrack$. Plugging in the value $1$ for each $z_1,\ldots,z_{\Delta - 1},z_n$ and rearranging the corresponding columns to the right hand side in $\bA\bz = \bb$ leaves us with the equations $z_{l} = 1$ for all $l\in\lbrace \Delta,\ldots,n-1\rbrace$. We conclude that $\bz = \bm{1}$ is the only integer-valued solution for $\bA\bx=\bb$ with $\bx\in\Z^n_{\geq 0}$.
	\end{proof}
	
	\section{An Algorithm for the \eqref{svp}}
	\label{sec_algorithm}
	This section describes how to turn the proof of
	Theorem~\ref{thm_shortest_vector} into an algorithm for solving
	\eqref{svp}. We assume $n\geq \Delta$ throughout this section.
	
	Let us begin by illustrating the idea. 
	Suppose that $\bB$ is an invertible
	full rank submatrix of $\bA$ with $\left|\det\bB\right| \leq \Delta$
	and that there exist signs $\sigma_1,\ldots,\sigma_n\in\lbrace -
	1,1\rbrace$ such that
	\begin{enumerate}
		\item $\bA\cdot\bB^{-1} \in \lbrack -
		1,1\rbrack^{m\times n}$ and
		\item $\bA\cdot\bB^{-1} \cdot
		\diag(\sigma_1,\ldots,\sigma_n)$ contains no row with entries $1$ and $-1$.
	\end{enumerate}
	Following the proof of Theorem~\ref{thm_shortest_vector}, it
	now suffices to check whether a column of $\bB^{-1}\cdot
	\diag(\sigma_1,\ldots,\sigma_n)$ is an integer vector or find two
	distinct columns whose difference is an integer vector. These operations can be
	performed in polynomial time and hence, we have found  a lattice vector of infinity norm equal to $1$. 
	
	These explanations show us that we need to find  a basis matrix $\bB$ and signs
	$\sigma_1,\ldots,\sigma_n$ that satisfy the two properties above.
	The proof of Theorem~\ref{thm_shortest_vector} also suggests
	that we should find a matrix $\bB$ with determinant $\left|\det\bB\right| =
	\Delta$. Unfortunately, the task of finding a submatrix $\bB$ with
	largest subdeterminant in polynomial time for fixed $\Delta$ is a major
	open problem. 
	
	We circumvent this issue by considering a sequence of
	invertible submatrices. The key insight is that either the current
	submatrix satisfies the first property or we can generate a new
	submatrix with strictly larger determinant in absolute value. This second scenario cannot appear too often since, otherwise, the algorithm returns a submatrix with determinant larger than $\Delta$ in absolute value.
	
	If we are in the situation that a basis matrix satisfies the first property, then we adapt an algorithm
	by Fukuda and Musitelli \cite{FukudaMusitelli2006} to find  a
	Camion basis in polynomial time, in a way that allows us to either construct a submatrix that
	satisfies both properties, or detect a new submatrix with larger
	determinant in absolute value.
	We summarize the steps of the algorithm that solves \eqref{svp} when $n\geq \Delta$  below; see Algorithm~\ref{algorithm_main}. For a matrix
	$\bC$ with rows indexed by $N$ and columns by $I$, we use the notation 
	$\rowsum(\bC_{k,\cdot}) := \sum_{i\in I} C_{k,i}$ and
	$\colsum(\bC_{\cdot,j}) := \sum_{i\in N}C_{i,j}$ for $k\in N$ and $j\in
	I$.
	
	\begin{algorithm}
		\caption{Polynomial Time Algorithm for the \eqref{svp}
			when $n\geq \Delta$}
		\label{algorithm_main}
		\small
		\begin{algorithmic}[1]
			\Statex \textbf{Input:} Full column rank
			matrix $\bA\in\Z^{m\times n}$ and $\Delta\in\N_{\geq 1}$ with
			$n\geq\Delta$.
			\Statex \textbf{Output:} Either
			$\by\in\bA\Z^n$ with $\Vert\by\Vert_\infty=1$ or a full rank submatrix
			of $\bA$ with determinant in absolute value greater than $\Delta$.
			\State Find $I\subseteq\lbrack m\rbrack$ with
			$|I|=n$ and $\det\bA_I\neq 0$. Set $\bW:=\bA$.
			\Loop
			\If{$\left|\det\bA_I\right|>\Delta$}
			\State\label{alg_return_large_determinant} \textbf{return} $\bA_I$.
			\EndIf
			\State Set $N:=\lbrack
			m\rbrack\setminus I$. Compute $\bC:=\bW_N\cdot\bW_I^{-1}$.
			\State Label the rows of $\bC$ by $N$
			and the columns of $\bC$ by $I$ in increasing order.
			\If{$|C_{k,p}|>1$ for some $k\in N$
				and $p\in I$}
			\State\label{alg_update_determinant} Set $I:=(I\setminus\{p\})\cup\{k\}$.
			\State \textbf{continue}.
			\EndIf
			\ForAll{$j\in I$ with
				$\colsum(\bC_{\cdot,j})<0$}
			\State\label{alg_column_multiplication} Multiply row $j$ of $\bW$ by $-1$ and update $\bC$.
			\EndFor
			\If{$\bC$ contains no row with entries
				$1$ and $-1$}
			\If{there exists a column
				$\br\in\Z^n$ of $\bW_I^{-1}$}
			\State\label{alg_return_integer_column} \textbf{return} $\by:=\bA\br$.
			\Else
			\State Find distinct
			columns $\br_1,\br_2$ of $\bW_I^{-1}$ with $\br_1-\br_2\in\Z^n$.
			\State\label{alg_return_integer_difference} \textbf{return} $\by:=\bA(\br_1-\br_2)$.
			\EndIf
			\EndIf
			\State Find $k\in N$ and $p,q\in I$
			with $C_{k,p}=1$ and $C_{k,q}=-1$.
			\If{$\rowsum(\bC_{k,\cdot})<1$}
			\State Set $j:=p$.
			\Else
			\State Set $j:=q$.
			\EndIf
			\State\label{alg_update_sum} Set
			$I:=(I\setminus\{j\})\cup\{k\}$.
			\If{$\rowsum(\bC_{k,\cdot})=1$}
			\State\label{alg_row_sum_one_multiplication} Multiply row $k$ of $\bW$ by $-1$.
			\EndIf
			\EndLoop
		\end{algorithmic}
	\end{algorithm}
	
	\begin{proof}[Proof of Theorem~\ref{thm_svp_algorithm}]
		The output is correct because we either return a full
		rank submatrix with determinant in absolute value larger than $\Delta$
		in Step~\ref{alg_return_large_determinant} or a lattice vector in
		Steps~\ref{alg_return_integer_column} or
		\ref{alg_return_integer_difference} whose $\ell_\infty$-norm is $1$ according
		to the proof of Theorem~\ref{thm_shortest_vector}.
		
		We claim that the algorithm terminates after
		finitely many steps. 
		If $m = n$, then the algorithm computes $\det\bA$ and either returns a submatrix with determinant larger than $\Delta$ in absolute value in Step~\ref{alg_return_large_determinant} or a shortest vector in Steps~\ref{alg_return_integer_column} and \ref{alg_return_integer_difference}. The loop is only visited once. So suppose that $m > n$. 
		To see that the algorithm terminates, observe that each iteration of the
		loop either returns something or updates the set $I$. We only update
		the set $I$ in Steps~\ref{alg_update_determinant} and
		\ref{alg_update_sum}. The update of the set $I$ in
		Step~\ref{alg_update_determinant} ensures that the updated set
		$(I\setminus\lbrace p\rbrace)\cup\lbrace k\rbrace$ gives
		\begin{align*}
			\frac{\left|\det\bA_{(I\setminus\lbrace
					p\rbrace)\cup\lbrace k\rbrace}\right|}{\left|\det\bA_I\right|} =
			\left|(\bA \cdot \bA^{-1}_I)_{k,p}\right| = \left|C_{k,p}\right| > 1,
		\end{align*}
		where the first equation holds due to
		Lemma~\ref{lemma_submatrix_determinants} and the second because
		$\bC$ differs from $\bA_N\cdot \bA^{-1}_I$ only by multiplication of
		rows and columns by $-1$. So integrality implies that
		$\left|\det\bA_{(I\setminus\lbrace p\rbrace)\cup\lbrace
			k\rbrace}\right| \geq \left|\det\bA_I\right| + 1$, a strict increase of
		the determinant in absolute value.
		
		The update in Step~\ref{alg_update_sum} does not
		increase the determinant in absolute value. In fact, since $|C_{k,j}| =
		1$, we get $\left|\det\bA_{(I\setminus\lbrace j\rbrace)\cup\lbrace
			k\rbrace}\right| = \left|\det\bA_I\right|$ from the argument above. So
		the determinant remains the same in absolute value. However, as in the
		proof of Fukuda and Musitelli \cite[Proof of
		Theorem~1]{FukudaMusitelli2006}, we increase the total sum of the
		matrix $\bC$: Let $\ssum(\bC) := \sum_{i\in N}\sum_{j\in I} C_{i,j}$ and
		let $\bC'$ be the matrix that is obtained after an update in
		Step~\ref{alg_update_sum}. Then, by \cite[Proof of
		Theorem~1]{FukudaMusitelli2006}, we have the following update formula
		\begin{align*}
			\ssum(\bC')-\ssum(\bC) = \frac{1 -
				\rowsum(\bC_{k, \cdot})}{C_{k,j}} \left(1 +
			\colsum(\bC_{\cdot,j})\right).
		\end{align*}
		Suppose first that $\rowsum(\bC_{k,\cdot}) \neq 1$.
		Since we update on $C_{k,p} = 1$ if $\rowsum(\bC_{k,\cdot}) < 1$ and
		$C_{k,q} = -1$ otherwise, and since  $\colsum(\bC_{\cdot,j}) \geq 0$ by
		Step~\ref{alg_column_multiplication}, we get
		\begin{align*}
			\ssum(\bC')-\ssum(\bC) \geq \left|1 -
			\rowsum(\bC_{k, \cdot})\right|\geq \frac{1}{\left|\det\bA_I\right|}\geq
			\frac{1}{\Delta}.
		\end{align*}
		So the total sum increases by at least $1/\Delta$.
		However, when $\rowsum(\bC_{k,\cdot}) = 1$, we get no improvement as
		$\ssum(\bC')-\ssum(\bC) = 0$. This is why we multiply the $k$-th row of
		$\bW$ by $-1$ in the additional
		Step~\ref{alg_row_sum_one_multiplication}. This creates a new matrix
		$\bC''$. We claim that $\bC''$ has a larger total sum than
		$\ssum(\bC)$. To see this, we use the following update formula
		\cite[Proof of Theorem~1]{FukudaMusitelli2006} (noting that $k$ has
		entered $I$ in Step~\ref{alg_update_sum} and therefore corresponds
		to a column of $\bC'$):
		\begin{align*}
			\colsum(\bC'_{\cdot,k}) = \frac{1}{C_{k,q}}
			\left(1 + \sum_{i\in N\backslash \lbrace k\rbrace}C_{i,q}\right) = -
			\left(1 + \sum_{i\in N\backslash \lbrace k\rbrace}C_{i,q}\right).
		\end{align*}
		By Step~\ref{alg_column_multiplication}, we have
		$C_{k,q} + \sum_{i\in N\backslash \lbrace k\rbrace}C_{i,q} =
		\colsum(\bC_{\cdot, q})\geq 0$. So $\sum_{i\in N\backslash \lbrace
			k\rbrace}C_{i,q}\geq -C_{k,q} = 1$. This yields
		\begin{align*}
			\colsum(\bC'_{\cdot,k}) = - \left(1 +
			\sum_{i\in N\backslash \lbrace k\rbrace}C_{i,q}\right) \leq - 2.
		\end{align*}
		Now observe that $\bC''$ arises from $\bC'$ by
		multiplying the column labelled $k$ by $-1$. The inequality above
		implies that the total sum has increased by at least $4$. In summary,
		after the update in Step~\ref{alg_update_sum} and, if necessary,
		the multiplication in Step~\ref{alg_row_sum_one_multiplication}, the
		total sum has increased by at least $1/\Delta$. The total sum lies
		between $0$ and $(m-n)n$ because the column sums are non-negative due to
		Step~\ref{alg_column_multiplication} and $\bC\in\lbrack -
		1,1\rbrack^{(m-n)\times n}$. So, for a fixed determinant value, there
		are at most $\mathcal{O}(\Delta(m-n)n)$ updates. Since the algorithm
		allows for at most $\Delta$ determinant values, it terminates after
		$\mathcal{O}(\Delta^2(m-n)n)$ iterations.
		
		To improve the running time, we can update $\bC$ and
		$\left|\det\bA_I\right|$ by simple pivot formulas similar to the ones
		used in \cite{FukudaMusitelli2006}. Then, each iteration can be done in
		$\mathcal{O}(1 + (m-n)n)$ arithmetic operations, where the additive term $1$ addresses the case when $m = n$. To select the initial basis and compute the
		initial inverse, the algorithm requires a one-time cost of
		$\mathcal{O}(mn^2)$ arithmetic operations via Gaussian elimination.
		Thus the algorithm uses $\mathcal{O}(mn^2 + \Delta^2(m-n)^2n^2) =
		\mathcal{O}(\Delta^2m^2n^2)$ arithmetic operations.
	\end{proof}

	\noindent
	\textbf{Acknowledgements.} The authors thank the reviewers of the extended abstract for their valuable comments that led to the improved bound presented in \cite{svp_delta} and Joe Paat for helpful remarks on an earlier draft of the extended abstract. 
	\vspace*{.3cm}
	\newline
	\noindent
	\textbf{AI Disclosure.} GPT-6 Astra generated a proof of the improved upper bound on \eqref{supremum} and the new algorithm, based on the extended abstract \cite{svp_delta}. The authors simplified the proof and algorithm, and contextualized the link to Camion bases. All other results and the writing are due to the authors.

	\bibliography{references}

\end{document}